\documentclass[12pt]{amsart}
\usepackage{amsmath,amsthm,amscd,amssymb,bbm,graphicx}
\usepackage{floatflt,setspace, wrapfig}
\usepackage{graphics,multicol}%changepage}
\usepackage{mathrsfs} %gives fance script A font for attractor
\usepackage{epstopdf}
\usepackage{geometry}

\newtheorem{thrm}{Theorem}%[section]

\newtheorem*{cor}{Corollary}
\newtheorem{prop}{Proposition}%[section]
\newtheorem{lem}{Lemma}%[section]
\newtheorem*{cdn}{Condition}

\newcommand{\maps}{\rightarrow}
\newcommand{\nats}{\mathbb{N}_0}
\newcommand{\ind}{\mathbbm{1}}
\newcommand{\ev}{\mathbb{E}}
\newcommand{\ball}{B_{\rho}}
\newcommand{\diam}{\mbox{diam }}

\providecommand{\abs}[1]{\lvert \, #1 \, \rvert}
\providecommand{\Abs}[1]{\biggl\lvert \, #1 \, \biggr \rvert}
\providecommand{\norm}[1]{\lVert \, #1 \, \rVert}
\providecommand{\floor}[1]{\left\lfloor \, #1 \, \right\rfloor}
\providecommand{\roof}[1]{\left\lceil \, #1 \, \right\rceil}

\newcommand{\bea}[1]{\begin{eqnarray}\label{#1}}
\newcommand{\eea}{\end{eqnarray}}

\title[Poisson Law for returns of  Maps  on Compact Manifolds]{Poisson Law for returns of  Maps  on Compact Manifolds}
 \date{\today}
 
\begin{document}
 \maketitle
\authors{Nicolai T A Haydn\footnote{Department of Mathematics, University of Southern California,
Los Angeles, 90089-2532. E-mail: {\tt \email{nhaydn@usc.edu}}.},
Fan Yang\footnote{Department of Mathematics, University of Southern California,
Los Angeles, 90089-2532. E-mail: {\tt \email{fizbayang@gmail.com}}.}}

%\tableofcontents

%%%%%%%%%%%%%%%%%%%%%%%%%%%%%%%%%%%%%%%%%%%%
%%%%%%%%%%%%%%%%%%%%%%%%%%%%%%%%%%%%%%%%%%%%
%%%%%%%%%%%%%%%%%%        ABSTRACT
\begin{abstract}

We consider invariant measures of maps on manifolds whose correlations decay
at a sufficient rate and which satisfy a geometric contraction property.
We then prove the that the limiting distribution of returns to geometric balls
is Poissonian. This does not assume an tower construction.
The decay of correlations is used to show that the independence generated 
results in the Poisson distribution for returns that are sufficiently separated.
A geometric contraction property is then used to show that short return times
have a vanishing contribution to the return times distribution. We then also 
show that the set of very short returns which are of a small linear order
of the logarithm of the radius of the balls has a vanishing measure. 
We obtain error terms which decay polynomially in the logarithm of the 
radius. We also obtain a extreme value law for such systems. 

\end{abstract}

\tableofcontents

%%%%%%%%%%%%%%%%%%%%%%%%%%%%%%%%%%%%%%%%%%%%
%%%%%%%%%%%%%%%%%%%%%%%%%%%%%%%%%%%%%%%%%%%%
%%%%%%%%%%%%%%%%%%             INTRODUCTION
\section{Introduction}

The limiting distribution of higher order return times in dynamics 
goes back to Doeblin~\cite{Doe} who establised the Poisson distribution in the 
the limit for the Gauss map at the origin. In more recent times there has
been a large number of results for returns to cylinder sets where
certain mixing properties are assumed. Pitskel~\cite{Pit} proved for
Axiom~A maps and equlibrium measures for H\"older continuous 
potentials that unions of cylinder sets have in the limit Poisson distributed
return times. An approximation argument allowed him then to also
deduce the same result for balls in the case of an Axiom~A map on
the two dimensional torus. He showed that the moments converge 
and then invoked a theorem of Sevast'yanov to conclude that 
the return times are Poissonian in the limit. Other subsequent results 
like by Denker~\cite{Den} which is along similar lines and in~\cite{HSV}
which considers parabolic maps on the interval extended those
results to more general settings. For rational maps this was done in~\cite{H00}
which allowed approximations for balls if the dimension of the measure was
not too large. More recently in~\cite{Ab3,AV1} 
results were obtain for $\phi$-mixing and $\alpha$-mixing systems
along a nested sequence of cylinder sets. Using the method of Chen and
Stein the Poisson distribution was established in~\cite{HP10} 
along sequences of unions of cylinders for $\phi$-mixing systems
over countable alphabets. This also allowed for approximations for 
balls under some favourable conditions. In~\cite{HY14} this was 
extended to $\alpha$-mixing systems and applied to the returns to 
Bowen balls. For a review see e.g.~\cite{H13}.

For returns to geometric balls apart from Pitskel's result~\cite{Pit} from 1990
all results are quite recent except for the one by Pitskel and on intervals
where approximations by cylinder sets can easily be used.
 For systems that can be modelled by Young 
towers, Chazottes and Collet~\cite{CC13} proved the limiting 
distribution to be Poissonian if the decay of correlations are  exponential
and the unstable manifold is one dimensional. This was generalised to 
polynomially decaying correlations and arbitrary dimensions in~\cite{HW14}.
In this case the speed of convergence is polynomial in the logarithm of the 
radius of the return ball.
A similar result without speed of convergence was obtained in~\cite{PS}
by using the Lebesgue density theorem.
In this paper we provide a limiting result for maps on manifolds whose 
correlations decay and which satisfy a certain uniform contraction 
property (Assumption~(V) below).

%%%%%%%%%%%%%%%%%%%%%%%%%%%%%%%%%%%%%%%%%%%%%%%%%%%%%%%%%%%%%%%%%%%%%%%
%%%%%%%%%%%%%%%%%%%%%%%%%%%%%%%%%%%%%%%%%%%%%%%%%%%%%%%%%%%%%%%%%%%%%%% Background
\section{Assumptions and main results} \label{assumptions}

Let $M$ be a manifold and $T:M\to M$ a map with the properties described below in
the assumptions. Let $\mu$ be a $T$-invariant probability measure on $M$.

For a ball $\ball(\mathsf{x})\subset M$ we define the counting function
\vspace{-0.2cm}
\begin{equation*}
\xi^{t}_{\rho,\mathsf{x}}(x)=\sum_{n=0}^{\floor{t/\mu(\ball(\mathsf{x}))}-1} {\ind_{\ball(\mathsf{x})}} \circ T^n(x).
\end{equation*}
which tracks the number of visits a trajectory of the point $x \in M$ makes to the ball $\ball(\mathsf{x})$
on an orbit segment of length $N=\floor{t/\mu(\ball(\mathsf{x}))}$, where $t$ is a positive
parameter.
(We often omit the sub- and superscripts and simply use $\xi(x)$.)

Let $\Gamma^u$ be a collection of unstable leaves $\gamma^u$
and  $\Gamma^s$ a collection of stable leaves $\gamma^s$. We assume
that $\gamma^u\cap\gamma^s$ consists of a single point for all $(\gamma^u,\gamma^s)\in\Gamma^u\times\Gamma^s$. 
The map $T$ contracts along the stable leaves and similarly $T^{-1}$ 
contracts along the unstable leaves.

For an unstable leaf $\gamma^u$ denote by $\mu_{\gamma^u}$ the 
disintegration of $\mu$ to the $\gamma^u$. We assume that $\mu$ 
has a product like decomposition  
$d\mu=d\mu_{\gamma^u}d\upsilon(\gamma^u)$,
where $\upsilon$ is a transversal measure. That is, if $f$ is a function on $M$ then
$$
\int f(x)\,d\mu(x)
=\int_{\Gamma^u} \int_{\gamma^u}f(x)\,d\mu_{\gamma^u}(x)\,d\upsilon(\gamma^u)
$$

If $\gamma^u, \hat\gamma^u\in\Gamma^u$ are two unstable leaves then the holonomy map $\Theta:\gamma^u\cap\Lambda\to \hat\gamma^u\cap\Lambda$ is defined by $\Theta(x)=\hat\gamma^u\cap\gamma^s(x)$ for $x\in\gamma^u\cap\Lambda$,  where $\gamma^u(x)$ be the local unstable leaf through $x$.

Let us denote by 
$J_n=\frac{dT^n\mu_{\gamma^u}}{d\mu_{\gamma^u}}$
the Jacobian of the map $T^n$ with respect to the measure $\mu$
in the unstable direction.

Let $\gamma^u$ be a local unstable leaf.
Assume there exists $R>0$ and  for every $n\in\mathbb{N}$ finitely many
 $y_k\in T^n\gamma^u$ so that 
$T^n\gamma^u\subset\bigcup_k B_{R,\gamma^u}(y_k)$, 
where $B_{R,\gamma^u}(y)$ is the embedded $R$-disk centered at $y$ in the unstable leaf
 $\gamma^u$.
Denote by $\zeta_{\varphi,k}=\varphi(B_{R,\gamma^u}(y_k))$ where $\varphi\in \mathscr{I}_n$
 and $\mathscr{I}_n$ denotes the  inverse branches of $T^n$. 
 We call $\zeta$ an $n$-cylinder.
Then there exists a constant $L$ so that the number of overlaps
$N_{\varphi,k}=|\{\zeta_{\varphi',k'}: \zeta_{\varphi,k}\cap\zeta_{\varphi',k'}\not=\varnothing,
\varphi'\in\mathscr{I}_n\}|$
is bounded by $L$ for all $\varphi\in \mathscr{I}_n$ and for all $k$ and $n$. 
This follows from the fact that
$N_{\varphi,k}$ equals $|\{k': B_{R,\gamma^u}(y_k)\cap B_{R,\gamma^u}(y_{k'})\not=\varnothing\}|$ 
which is uniformly bounded by some constant $L$.

%%%%%%%%%%%%%%%%%%%%%%%%%%%%%%%%%%%%%%%%%%%%%%%
%%%%%%%%%%%%%%%%%%%%%%%%%%%%%%%%%%%%%%%%%%%%%%%
%%%%%%%%%%%%              ASSUMPTIONS
We make the following assumptions:\\
(I) {\em Decay of correlations:} There exists a decay function $\lambda(k)$ so that 
$$
\left|\int_MG(H\circ T^k)\,d\mu-\mu(G)\mu(H)\right|
\le \lambda(k)\|G\|_{Lip}\|H\|_\infty\qquad\forall k\in\mathbb{N},
$$
for functions $H$ which are constant on 
local stable leaves $\gamma^s$ of $T$.\\
(II) {\em Dimension:} There exist $0<d_0<d_1$ such that $\rho^{d_0}\ge\mu(\ball)\ge\rho^{d_1}$.\\
(III) {\em Unstable dimension:} There exists a $u_0$ so that 
$\mu_{\gamma^u}(B_\rho(x))\le C_1\rho^{u_0}$ for all $\rho>0$ small enough
and for almost all $x\in \gamma^u$, every unstable leaf $\gamma^u$.\\
(IV)  {\em Distortion:} We  require that 
$\frac{J_n(x)}{J_n(y)}=\mathcal{O}(\omega(n))$ for all $x,y\in\zeta$ and $n$, where 
$\zeta$ are $n$-cylinders in unstable leaves $\gamma^u$ and $\omega(n)$ is a non-decreasing 
sequence.\\
(V) {\em Contraction:} There exists a function $\delta(n)\to0$ which decays at least summably polynomially, i.\,e.\, $\delta(n) = \mathcal{O}(n^{-\kappa})$ with $\kappa > 1$, so that 
$\diam\zeta\le \delta(n)$ for all $n$-cylinder $\zeta$ and all $n$.\\
(VI) {\em Annulus condition:} Assume that for some $\eta, \beta>0$:
$$
\frac{\mu(B_{\rho+r}\setminus B_{\rho-r})}{\mu(\ball)} 
= \mathcal{O}(r^\eta\rho^{-\beta})
$$
for every $r< \rho_0$ for some $\rho_0<\rho$ (see remark below).

\vspace{2mm}

\noindent For a positive parameter
$\mathfrak{a}$ define the set
\begin{equation} \label{defM_rho,J}
\mathcal{V}_\rho(\mathfrak{a}) = \{\mathsf{x} \in M: \ball(\mathsf{x}) \cap T^{n}\ball(\mathsf{x}) \ne \varnothing \text{ for some } 1 \leq n < \mathfrak{a}\abs{\log\rho}\},
\end{equation}
where $\rho>0$. The set $\mathcal{V}_\rho$ represents the points within $M$ with very short return times.

%%%%%%%%%%%%%%%%%%%%%%%%%%%%%%%%%%%%%%%%%%%%%%%
%%%%%%%%%%%%%%%%%%%%%%%%%%%%%%%%%%%%%%%%%%%%%%%
%%%%%%%%%%%%             SUBSECTION: POISSON DISTRIBUTED RETURN TIMES
%%%%%%%%%%%%%%%%%%%%%%%%%%%%%%%%%%%%%%%%%%%%%%%
\subsection{Return times are Poisson distributed}
\begin{thrm}
Assume that the map $T: M\to M$ satisfies the assumptions (I)--(VI) where $\lambda(k)$ decays at 
least polynomially with power $p>\frac{\frac\beta\eta+d_1}{d_0}$. Moreover we assume that
$d_0(\kappa\eta-1)>\beta$ and $\kappa u_0>1$.
If $\delta(j)$ decays polynomially with power $\kappa$ and $\omega(j)\sim j^{\kappa'}$ 
for some $\kappa'\in[0,\kappa u_0-1)$.

Then 
$$
\mathbb{P}\!\left(\xi_{\rho,\mathsf{x}}=r\right)= e^{-t}\frac{t^r}{r!}+\mathcal{O}(\abs{\log\rho}^{- \sigma})
$$
for all $\mathsf{x}\not\in\mathcal{V}_\rho(\mathfrak{a})$ for some positive $\mathfrak{a}$,
where $\sigma=\kappa u_0-\kappa'-1$ is positive.
Moreover, there exists an $\mathfrak{a}>0$ so that 
$$
\mu(\mathcal{V}_\rho(\mathfrak{a}))=\mathcal{O}(\abs{\log\rho}^{- \sigma}).
$$

If $\delta(n)=\mathcal{O}(\vartheta^n), \vartheta<1$ is exponential then 
$$
\mathbb{P}\!\left(\xi_{\rho,\mathsf{x}}=r\right)= e^{-t}\frac{t^r}{r!}
+\mathcal{O}(t\rho^{u_0\mathfrak{a}\abs{\log\vartheta}}).
$$
\end{thrm}

The proof of the first part of the theorem is given in the next section. The bound on
the size of the very short return set is given in Section~\ref{veryshort}.

\vspace{2mm}

\noindent  {\em Remark~1:} The standard case of bounded distortion corresponds
to the value $\kappa'=0$. Then the rate of convergence is $\sigma=\kappa u_0-1$.

\vspace{2mm}

\noindent  {\em Remark~2:} If $\mu$  has dimension $d$ then the condition on the decay of 
correlation is $p>1+\frac\beta{d\eta}$.

\vspace{2mm}

\noindent  {\em Remark~3:} For an absolutely continuous measure $\beta$ and $\beta$ 
can be chosen arbitrarily close to $1$, $d_0$ and $d_1$ are arbitrarily close to $D=\dim M$; 
thus the requirement for $p$ is to be larger than $\frac1D+1$ and $\kappa>\frac1{D}+1$
as $u_0$ is arbitrarily close to $D$.

\vspace{2mm}

\noindent  {\em Remark~4:}  In the annulus condition~(IV) we require that $r<\rho_0$ where
according to Section~\ref{total_error} $\rho_0=\mathcal{O}(\rho^{v\kappa})$ and $v<d_0$
can be arbitrarily close to $d_0$.

%%%%%%%%%%%%%%%%%%%%%%%%%%%%%%%%%%%%%%%%%%%%%%%
%%%%%%%%%%%%%%%%%%%%%%%%%%%%%%%%%%%%%%%%%%%%%%%
%%%%%%%%%%%%             SUBSECTION: POISSON DISTRIBUTED RETURN TIMES
%%%%%%%%%%%%%%%%%%%%%%%%%%%%%%%%%%%%%%%%%%%%%%%
\subsection{Extremal Values Distribution}

 We take a point $z \in M$ and define
$$
\varphi(x) = g(\mu(B_{d(x,z)}(z)))
$$
where $g$ is a function from $M$ to $\mathbb{R}\cup\{+\infty\}$ with the following properties: $g$ is strictly decreasing in a neighborhood of $0$; $0$ is a global maximum for $g$; $g$ satisfies one of the following three properties:

\noindent Type 1. There exist some strictly positive function $p$ such that
$$
g^{-1}\left(g(\frac{1}{n})+yp(g(\frac{1}{n}))\right)=(1+\varepsilon_n)\frac{e^{-y}}{n}.
$$
with $\epsilon_n \to 0$ as $n \to \infty$ for all $y \in \mathbb{R}$.

\noindent Type 2. $g(0) = +\infty$ and there exist $\beta>0$ such that 
$$
g^{-1}\left(g(\frac{1}{n})y\right) = (1+\varepsilon_n)\frac{y^{-\beta}}{n}
$$
for all $y>0$.

\noindent Type 3. $g(0) = D<+\infty$ and there exist $\gamma>0$ such that 
$$
g^{-1}\left(D-g(\frac{1}{n})y\right) =  (1+\varepsilon_n) g^{-1}\left(D-g(\frac{1}{n})\right)\frac{y^{\gamma}}{n}
$$
for all $y>0$.\\
Examples of functions satisfying the three types are $g_1(x) = -\log x$, $g_2(x) = x^{-\frac{1}{\beta}}$ and $g_3(x) = D-x^{\frac{1}{\gamma}}$.

We put $X_n = \varphi \circ T^n$ and $M_n = \max\{X_k: 0 \le k \le n-1\}$. We also write
$$
M_{j,n} = \max\{X_j, \cdots, X_{j+n-1}\}.
$$
Let $\{\hat{X}_n\}$ be the a stationary, independent process such that $\hat{X}_0$ has the same distribution as $X_0$. Denote by $\hat{M}_n$ the corresponding maxima of $\{\hat{X}_n\}$. From the extreme value theory of stationary, independent processes we know that under proper linear normalization, $a_n(\hat{M}_n - b_n)$ converges to one of the following three limits:

\noindent Type 1. 
$$
G(x) = e^{-e^{-x}}, \hspace{2cm} -\infty < x < \infty.
$$

\noindent Type 2. 
$$
G(x) = \begin{cases}
0 \hspace{2cm} \text{if } x \le 0\\
e^{-x^{-\beta}} \hspace{1cm}\text{ if } x >0, \text{ for some } \beta>0.
\end{cases}
$$

\noindent Type 3. 
$$
G(x) = 
\begin{cases}
e^{-(-x)^{\gamma}} \hspace{1cm} \text{if }x\le 0 \text{ for some} \gamma >0\\
1 \hspace{2cm} \text{if }x>0.
\end{cases}
$$

Now we state the Theorem.

\begin{thrm}\label{EVL}
Assume that the map $T: M\to M$ satisfies the assumptions (I)--(VI) where $\lambda(k)$ decays at 
least polynomially with power $p>\frac{\beta}{\eta d_1}+1 $. Moreover we assume that
$\kappa\eta-1>\beta$ and $\kappa u_0>2$ if $\delta(j)$ decays polynomially with power 
$\kappa$ and $\omega(j)\sim j^{\kappa'}$ 
for some $\kappa'\in[0,\kappa u_0-3)$. Then we have Type $i$ extreme value law for observables $g_i$ with type $i$, $i=1,2,3.$
\end{thrm}

%%%%%%%%%%%%%%%%%%%%%%%%%%%%%%%%%%%%%%%%%%%%%
%%%%%%%%%%%%%%%%%%%%%%%%%%%%%%%%%%%%%%%%%%%%%
%%%%%%%%%%  SECTION: PROOF OF THE THEOREM
\section{Proof of Theorem 1}

%%%%%%%%%%%%%%%%%%%%%%%%%%%%%%%%%%%%%%%%%%%%
%%%%%%%%%%%%%%%%%%%%%%%%%%%%%%%%%%%%%%%%%%%%
%%%%%%%%%%%%%%%%%%%%%%%%%%% Set up the proof
\subsection{Poisson approximation of the return times distribution} \label{set_up_T1}
To prove Theorem~1 we will employ the Poisson approximation theorem from Section~\ref{poisson}.
Let $\mathsf{x}$ be a point in the phase space and $\ball := \ball(\mathsf{x})$ for $\rho>0$.
 Let $X_n=\ind_{\ball} \circ T^{n-1}$, then we put $N = \floor{t/\mu(\ball)}$,
 where $t$ is a positive parameter. We write $S_a^b=\sum_{n=a}^bX_n$ (and $S=S_1^N$).
Then for any $2 \leq \Delta \leq N$ ($C_3$ from Section~\ref{poisson})
\begin{equation} \label{errorUSE}
\Abs{\mathbb{P}(S=k) - \frac{t^k}{k!} \, e^{-t}} \; 
\leq \; C_3 ( N(\mathcal{R}_1+\mathcal{R}_2) + \Delta \, \mu(\ball)),
\end{equation}
where
\begin{align*}
\mathcal{R}_1 
&= \sup_{\substack {0 <j<N-\Delta \\ 0<q<N-\Delta-j}} \left|\ev(\ind_{\ball} \ind_{S_{\Delta+1}^{N-j}=q})
- \mu(\ball) \, \ev(\ind_{S_{\Delta+1}^{N-j}=q})\right| \\
\mathcal{R}_2 &= \sum_{n=1}^{\Delta-1} \ev(\ind_{\ball} \; \ind_{\ball} \circ T^n).
\end{align*}
Since we restrict to the complement of the set $\mathcal{V}_\rho$ (cf.~\eqref{defM_rho,J})
we have from now on
$$
\mathcal{R}_2 = \sum_{n=J}^{\Delta-1} \mu(\ball \cap T^{-n} \ball),
$$
where  $J = \floor{\mathfrak{a} \, \abs{\log \rho}}$.
Since $\mathbb{P}(S=k)=0$ for $k>N$ we obtain 
\begin{equation} \label{yay!_k>N}
\Abs{\mathbb{P}(S=k) - \frac{t^k}{k!} \, e^{-t}} = \frac{t^k}{k!} \, e^{-t} 
\leq c_1\abs{\log \rho}^{-\kappa u_0} \qquad \forall k>N
\end{equation}
using the fact that $\mu(B_\rho)\lesssim\rho^{d_0}$ and for $\rho$ sufficiently small.

\vspace{0.5cm}
\noindent We now proceed to estimate the error between the distribution of $S$ and a Poissonian for
$k \leq N$ based on Theorem~\ref{helperTheorem}.

%%%%%%%%%%%%%%%%%%%%%%%%%%%%%%%%%%%%%%%%%%%
%%%%%%%%%%%%%%%%%%%%%%%%%%%%%%%%%%%%%%%%%%%
%%%%%%%%%%%%%%%%% ESTIMATING R1
\subsection{Estimating $\mathcal{R}_1$} \label{est_R1_section}

By invariance of the measure $\mu$ we can also write
$$
\mathcal{R}_1 = \sup_{\substack {0 <j<N-\Delta \\ 0<q<N-\Delta-j}}\left| \mu(\ball \cap T^{-\Delta}\{S_{1}^{N-j-\Delta}=q\}) - \mu(\ball) \, \mu(\{S_{1}^{N-j-\Delta}=q\}) \right|.
$$
We now use the decay of correlations~(I) to obtain an estimate for
$\mathcal{R}_1$. Approximate $\ind_{\ball}$ by Lipschitz functions
from above and below as follows:
\begin{equation*}
\phi(x) =
\begin{cases}
1 & \text{on $\ball$} \\
0 & \text{outside $B_{\rho + \delta \rho}$}
\end{cases}
\hspace{0.7cm} \text{and} \hspace{0.7cm}
\tilde{\phi}(x) =
\begin{cases}
1 & \text{on $B_{\rho - \delta \rho}$} \\
0 & \text{outside $\ball$}
\end{cases}
\end{equation*}
with both functions linear within the annuli. The Lipschitz norms of both $\phi$ and $\tilde{\phi}$ are equal to $1/\delta\rho$ and $\tilde{\phi} \leq \ind_{\ball} \leq \phi$.

We obtain
\begin{align*}
\mu(\ball \cap \{S_{\Delta}^{N-j}=q\}) - \mu(\ball) \, \mu(\{S_{1}^{N-j-\Delta}=q\})\hspace{-3cm} \\
& \leq \int_M \phi \; \cdot \ind_{S_\Delta^{N-j}=q} \, d\mu - \int_M \ind_{\ball} \, d\mu \, \int_M \ind_{S_1^{N-j-\Delta}=q} \, d\mu \\[0.2cm]
& =X+Y
\end{align*}
where
\begin{align*}
X&=\left(\int_M \phi \, d\mu - \int_M \ind_{\ball} \, d\mu \right) \int_M \ind_{S_1^{N-j-\Delta}=q} \, d\mu\\
Y&=\int_M \phi \; (\ind_{S_\Delta^{N-j}=q} ) \, d\mu - \int_M \phi \, d\mu \, \int_M \ind_{S_1^{N-j-\Delta}=q} \, d\mu .
\end{align*}
The two terms $X$ and $Y$ are estimated separately.
The first term is estimated by:
$$
X \leq\int_M \ind_{S_1^{N-j-\Delta}=q} \, d\mu \, \int_M (\phi - \ind_{\ball}) \, d\mu
 \leq \mu(B_{\rho + \delta \rho} \setminus \ball).
$$
In order to estimate the second term $Y$ we use the decay of correlations and
have to approximate $\ind_{S_{1}^{N-j-\Delta}=q}$ by a function which is constant on local stable leaves.
For that purpose put
$$
\mathcal{S}_n
=\bigcup_{\substack{\gamma^s\\T^n\gamma^s\subset B_\rho}}T^n\gamma^s,
\hspace{6mm}
\partial\mathcal{S}_n
=\bigcup_{\substack{\gamma^s\\T^n\gamma^s\cap B_\rho\not=\varnothing}}T^n\gamma^s
$$
and
$$
\mathscr{S}_\Delta^{N-j}=\bigcup_{n=\Delta}^{N-j}\mathcal{S}_n,
\hspace{6mm}
\partial\mathscr{S}_\Delta^{n-j}=\bigcup_{n=\Delta}^{N-j}\partial\mathcal{S}_n.
$$
The set
$$
\mathscr{S}_\Delta^{N-j}(q)=\{S_{1}^{N-j-\Delta}=q\}\cap\mathscr{S}_\Delta^{N-j}
$$
is then a union of local stable leaves. This follows from the fact that by construction
$T^n y\in B_\rho$ if and only if $T^n\gamma^s(y)\subset B_\rho$.
We also have
$\{S_p^{N-j}=q\}\subset\tilde{\mathscr{S}}_\Delta^{N-j}(q)$
where the set $\tilde{\mathscr{S}}_\Delta^{N-j}(k)=\mathscr{S}_\Delta^{N-j}(k)\cup\partial\mathscr{S}_\Delta^{N-j}$
is a union of local stable leaves.

Denote by  $\psi_\Delta^{N-j}$ the characteristic function of $\mathscr{S}_\Delta^{N-j}(k)$
and by $\tilde\psi_\Delta^{N-j}$ the characteristic function of
$\tilde{\mathscr{S}}_\Delta^{N-j}(k)$. Then $\psi_\Delta^{N-j}$ and $\tilde\psi_\Delta^{N-j}$
are constant on local stable leaves and satisfy
$$
\psi_\Delta^{N-j}\le\ind_{S_{1}^{N-j-\Delta}=q}\le\tilde\psi_\Delta^{N-j}.
$$
Since $\{y:\psi_\Delta^{N-j}(y)\not=\tilde\psi_\Delta^{N-j}(y)\}\subset\partial\mathscr{S}_\Delta^{N-j}$
we need to estimate the measure of $\partial\mathscr{S}_\Delta^{N-j}$.

 By the contraction property
  $\mbox{diam}(T^n\gamma^s(y))\le\delta(n)$ and consequently
  $$
  \bigcup_{\substack{\gamma^s\\T^n\gamma^s\subset B_\rho}}T^n\gamma^s
  \subset B_{\rho+\delta(n)}\setminus B_{\rho-\delta(n)}
  $$
and therefore
$$
\mu(\partial\mathscr{S}_\Delta^{N-j})
\le\mu\left(\bigcup_{n=\Delta}^{N-j}T^{-n}\left(B_{\rho+\delta(n)}\setminus B_{\rho-\delta(n)}\right)\right)
\le\sum_{n=\Delta}^{N-j}\mu(B_{\rho+\delta(n)}\setminus B_{\rho-\delta(n)}).
$$
Hence, by assumption~(VI),
$$
\mu(\partial\mathscr{S}_\Delta^{N-j}) 
=\mathcal{O}(1) \sum_{n=\Delta}^{\infty}\frac{n^{-\kappa\eta}}{\rho^\beta}\mu(\ball)\\
=\mathcal{O}(\rho^{v(\kappa\eta-1) -\beta}) \mu(\ball)
$$
provided $\delta(n) = \mathcal{O}(n^{-\kappa})$ and $\Delta\sim\rho^{-v}$ for some positive $v > \frac{\beta}{\kappa\eta-1}$ which is determined in Section~\ref{total_error}  below. 
If we split $\Delta=\Delta'+\Delta''$
then we can estimate as follows:
\begin{align*}
Y&=\left|  \int_M \phi \; T^{-\Delta'}(\ind_{S_{\Delta''}^{N-j-\Delta'}=q}  ) \, d\mu
- \int_M \phi \, d\mu \, \int_M \ind_{S_1^{N-j-\Delta}=q} \, d\mu\right|\\
&%\hspace{5cm}
\le \lambda(\Delta')\|\phi\|_{Lip}\|\ind_{\tilde{\mathscr{S}}_{\Delta''}^{N-j-\Delta'}}\|_{\mathscr{L}^\infty}
+2\mu(\partial\mathscr{S}_{\Delta''}^{N-j}).
  \end{align*}
Hence
\begin{eqnarray*}
\mu(\ball \cap T^{-\Delta}\{S_{1}^{N-j-\Delta}=q\}) - \mu(\ball) \, \mu(\{S_{1}^{N-j-\Delta}=q\})
\hspace{-4cm}&&\\
&\leq &\frac{\lambda(\Delta/2)}{\delta \rho} + \mu(B_{\rho + \delta \rho\setminus \ball })
+\mathcal{O}(\rho^{v(\kappa\eta-1)-\beta})\mu(\ball)
\end{eqnarray*}
by taking $\Delta' = \Delta'' = \frac{\Delta}{2}$.
A similar estimate from below can be done using $\tilde\phi$. Hence 
\begin{equation} \label{R1est}
\mathcal{R}_1
\leq c_2\left(\frac{\lambda(\Delta/2)}{\delta\rho}
+\mu(B_{\rho +\delta \rho} \setminus B_{\rho -\delta \rho})\right)
+\mathcal{O}(\rho^{v(\kappa\eta-1)-\beta})\mu(\ball).
\end{equation}

In the exponential case when $\delta(n)=\mathcal{O}(\vartheta^n)$  we choose
$\Delta=s\abs{\log\rho}$ for some $s>0$ and obtain the estimate 
$$
\mathcal{R}_1
\leq c_2\left(\frac{\lambda(\Delta/2)}{\delta\rho}
+\mu(B_{\rho +\delta \rho} \setminus B_{\rho -\delta \rho})\right)
+\mathcal{O}(\rho^{s\abs{\log\vartheta}-\beta})\mu(\ball).
$$

%%%%%%%%%%%%%%%%%%%%%%%%%%%%%%%%%%%%%%%%%%%
%%%%%%%%%%%%%%%%%%%%%%%%%%%%%%%%%%%%%%%%%%%
%%%%%%%%%%%%%%%%% ESTIMATE OF R2 
\subsection{Estimating the  terms  $\mathcal{R}_2$}

We will estimate the measure of each of the summands comprising $\mathcal{R}_2$ individually.
We use the product form of the measures
 $\mu$.
For that purpose fix $j$ and and let $\gamma^u$ be an unstable local leaf through $B$.
Then we put 
$$
\mathscr{C}_j(B,\gamma^u)=\{\zeta_{\varphi,j}: \zeta_{\varphi,j}\cap B\not=\varnothing,\varphi\in \mathscr{I}_j\}
$$
 for the cluster of $j$-cylinders that covers the set $B$,
 where the sets $\zeta_{\varphi,k}$ are the pre-images of embedded $R$-balls in 
 $T^j\gamma^u$.
Then 
\begin{eqnarray*}
\mu_{\gamma^u}(T^{-j}\ball\cap \ball)
&\le&\sum_{\zeta\in\mathscr{C}_j(\ball,\gamma^u)}\frac{\mu_{\gamma^u}(T^{-j}\ball\cap \zeta)}{\mu_{\gamma^u}(\zeta)}\mu_{\gamma^u}(\zeta)\\
&\le&\sum_{\zeta\in \mathscr{C}_j(\ball,\gamma^u)}c_3\omega(j)
\frac{\mu_{T^j\gamma^u}(\ball\cap T^j\zeta)}
{\mu_{T^j\gamma^u}(T^j\zeta)}\mu_{\gamma^u}(\zeta)
\end{eqnarray*}
Since $\mu_{T^j\gamma^u}(T^j\zeta)=\mu_{T^j\gamma^u}(B_{R,\gamma^u}(y_k))$
 (for some $y_k$) is uniformly bounded from below, we obtain  
\begin{eqnarray*}
\mu_{\gamma^u}(T^{-j}\ball\cap\ball)
&\le& c_3\omega(j)\mu_{T^j\gamma^u}(\ball)
\sum_{\zeta\in \mathscr{C}_j(\ball,\gamma^u)}\mu_{\gamma^u}(\zeta)\\
&\le& c_3\omega(j)\mu_{T^j\gamma^u}(\ball)\,  L\,
\mu_{\gamma^u}\!\left(\bigcup_{\zeta\in \mathscr{C}_j(\ball,\gamma^u)}\zeta\right)
\end{eqnarray*}
Now, since $\diam\bigcup_{\zeta\in \mathscr{C}_j(\ball,\gamma^u)}\zeta
\le \delta(j)+\diam \ball\le c_1\delta(j)$ 
(as we can assume that $\rho<\delta(j)$) we obtain
$$
\mu_{\gamma^u}(T^{-j}\ball\cap \ball)\le c_4\omega(j)\mu_{T^j\gamma^u}(\ball)\delta(j)^{u_0}.
$$
Since   $d\mu=d\mu_{\gamma^u}d\upsilon(\gamma^u)$
we obtain
$$
\mu(T^{-j}\ball\cap \ball)\le c_5\omega(j)\mu(\ball)\delta(j)^{u_0}.
$$
Summing up the $\mu(T^{-j}\ball\cap \ball)$  over $j=J,\dots,\Delta-1$, we see that outside the set of 
forbidden ball centers $\mathcal{V}_\rho $ we get
\begin{equation}\label{R_2}
\mathcal{R}_2 = \sum_{j=J}^{\Delta-1} \mu(T^{-j}\ball\cap \ball)
 \leq c_5\sum_{j=J}^{\Delta-1} \omega(j)\delta(j)^{u_0}\mu(\ball)
\end{equation}

for some $c_5$. If we assume that $\omega(j)\sim j^{\kappa'}$ and
$\delta(j)$ decays polynomially with power $\kappa$ then
$$
\mathcal{R}_2 
 \leq c_6\mu(\ball)\sum_{j=J}^{\Delta-1} j^{\kappa'}j^{-\kappa u_0}
  \leq c_7J^{-\sigma}\mu(\ball)
$$
for   $\rho$ small enough, where $\sigma=\kappa u_0-\kappa'-1$ is positive by assumption.
If $\delta(j)$ decays super polynomially then 
$$
\mathcal{R}_2   \leq c_7J^{\kappa'}\delta(J)^{u_0}\mu(\ball).
$$

In the exponential case ($\delta(n)=\mathcal{O}(\vartheta^n)$)  we obtain 
$$
\mathcal{R}_2 = \mathcal{O}(\rho^{u_0\mathfrak{a}\abs{\log\vartheta}})\mu(\ball)
  $$
  as $J=\mathfrak{a}\abs{\log\rho}$.

%%%%%%%%%%%%%%%%%%%%%%%%%%%%%%%%%%%%%%%%%%%
%%%%%%%%%%%%%%%%%%%%%%%%%%%%%%%%%%%%%%%%%%%
%%%%%%%%%%%%%%%%% ESTIMATE OF R
\subsection{The total error}\label{total_error} For the total error  we now put $\delta\rho=\rho^w$
and $\Delta=\rho^{-v}$ where $v<d_0$ since $\Delta<\!\!<N$ and $N\ge \rho^{-d_0}$. Then $\lambda(\Delta)=\mathcal{O}(\Delta^{-p})=\mathcal{O}(\rho^{pv})$
and thus (in the polynomial case)
\begin{eqnarray*}
\left|\mathbb{P}(S=k)-\frac{t^k}{k!}e^{-t}\right|
&\le& Nc_1\left(\frac{\lambda(\Delta)}{\delta\rho}
+\mu(B_{\rho +\delta \rho} \setminus B_{\rho -\delta \rho})\right)+c_7J^{- \sigma}
+\mathcal{O}(\rho^{v(\kappa\eta-1)-\beta})\\
&\le&c_8\left(\rho^{vp-w-d_1}+\rho^{w\eta-\beta}+tJ^{- \sigma}+\rho^{v(\kappa\eta-1)-\beta}\right).
\end{eqnarray*}
We can choose $v<d_0$ arbitrarily close to $d_0$ and then require
 $d_0p-w-d_1>0$,  $w\eta-\beta>0$ and $d_0(\kappa\eta-1)-\beta>0$. We can choose $w>\frac\beta\eta$ arbitrarily
 close to $\frac\beta\eta$ and can
 satisfy all requirements if $p>\frac{\frac\beta\eta+d_1}{d_0}$ in the case when $\lambda$ decays polynomially with power $p$, i.e.\ $\lambda(k)\sim k^{-p}$.
 Hence
 $$
 \left|\mathbb{P}(S=k)-\frac{t^k}{k!}e^{-t}\right|\le C_1J^{-\sigma}
$$
for some $C_1$.

In the exponential case ($\delta(n)=\mathcal{O}(\vartheta^n)$)  we obtain 
$$
\mathbb{P}(S=k)-\frac{t^k}{k!}e^{-t}
=\mathcal{O}(1)
\left(\rho^{s\abs{\log\vartheta}-w-d_1}+\rho^{w\eta-\beta}+t\rho^{u_0\mathfrak{a}\abs{\log\vartheta}}
+\rho^{s\abs{\log\vartheta}-\beta}\right).
$$
Choosing $s$ and $w$ large enough, we obtain that the RHS is of order 
$\mathcal{O}(t\rho^{u_0\mathfrak{a}\abs{\log\vartheta}})$.
\qed

%%%%%%%%%%%%%%%%%%%%%%%%%%%%%%%%%%%%%%%%%%%%%%%%
%%%%%%%%%%%%%%%%%%%%%%%%%%%%%%%%%%%%%%%%%%%%%%%%
%%%%%%%%%%%%%%%%%% SECTION VERY SHORT RETURNS
\section{Very Short Returns}\label{VeryShortReturns}

\subsection{Assumptions} \label{ass_T2}

Let $(M, T)$ be a dynamical system equipped with a metric $d$. Assume that the map $T: M \maps M$ is a $C^2$-diffeomorphism. As at the start of the paper the set $\mathcal{V}_\rho \subset M$ is given by
\begin{equation*}
\mathcal{V}_\rho = \{\mathsf{x} \in M: \ball(\mathsf{x}) \cap T^{n}\ball(\mathsf{x}) \ne \varnothing \text{ for some } 1 \leq n < J\},
\end{equation*}
where $J = \floor{\mathfrak{a} \, \abs{\log \rho}}$ and $\mathfrak{a} = (4 \log A)^{-1}$ with
$$
A =\sup_\omega 
\left(\norm{DT}_{\mathscr{L}^\infty} + \norm{DT^{-1}}_{\mathscr{L}^\infty}\right)
$$
($A\ge2$).
%%%%%%%%%%%%%%%%%%%%%%%%%%%%%%%%%%%%%%%%%%
%%%%%%%%%%%%%%%%%%%%%%%%%%%%%%%%%%%%%%%%%%
%%%%%%%%%%%      ASSUMPTIONS BBBBBBBBBBBBBBBBBB
We will need the following assumptions:

\vspace{2mm}

\noindent(V1)  {\em Distortion:}
 We  require that 
$\frac{J_n(x)}{J_n(y)}=\mathcal{O}(\omega(n))$ for all $x,y\in\zeta$ and $n$, where 
$\zeta$ are $n$-cylinders on unstable leaves $\gamma^u$ and
$\omega(n)$ is a (non-decreasing) sequence.

\vspace{2mm}

\noindent(V2) {\em Contraction:}
 There exists a function $\delta(n)\to0$ so that 
$\diam\zeta\le \delta(n)$ for all $n$-cylinder $\zeta$ and all $n$ and $\omega$.

\vspace{2mm}

\noindent (V3) {\em Geometric regularity of the measure on the unstable leaves:}
%%%%%%%%%%%%%%%%%%%%%%%%%%%%%%%%%%%%%%%%%%%%%%% B3
Assume there exists $u_0>0$ such that
$$
\mu_{\gamma^u}(\ball(\mathsf{x})) \leq \rho^{u_0}
$$
for all $\mathsf{x}$, unstable leaves $\gamma^u$ and $\rho$ small enough.

%%%%%%%%%%%%%%%%%%%%%%%%%%%%%%%%%%%%%%%%%%
%%%%%%%%%%%%%%%%%%%%%%%%%%%%%%%%%%%%%%%%%%
%%%%%%%%%%%     SUBSECTION: VERY SHORT RETURNS ESTIMATE
\subsection{Estimate on the measure of $\mathcal{V}_\rho$} \label{veryshort}

Now we can show that the set of centres where small balls have very short returns is small.
To be precise we have the following result:

%%%%%%%%%%%%%%%%%%%%%%%%%%%%%%%%%%%%%%%%%%%%
%%%%%%%%%%%%%%%%%%%%%%%%%%%%%%%%%%%%%%%%%%%%
%%%%%%%%%%%  PROPOSITION: VERY SHORT RETURNS
\begin{prop}\label{prop.short.returns}
Assume that the map $T:M\to M$ satisfies the assumptions (V1)--(V3).
Then there exist constants $C_{2}>0$ such that for all $\rho$ small enough
$$
\mu(\mathcal{V}_\rho)\le \frac{C_2}{\abs{\log\rho}^{ \sigma}}
$$
where 
 $\sigma=\kappa u_0-\kappa'-1$ if $\delta$ decays polynomially with power $\kappa>1$
 and $\omega$ grows polynomially with power $\kappa'\ge0$ assuming $\sigma>0$.
 
 If $\delta$ decays exponentially and $\limsup_{n\to\infty}\frac{\log\log\omega(n)}{\log n}<\frac12$
 then the error term on the RHS is $\mathcal{O}(\delta(\abs{\log \rho})^{u_0})$.
 
\end{prop}

\begin{proof} We follow the proof of Proposition~5.1 of~\cite{HW14} which modelled after 
 Lemma~4.1 of~\cite{CC13}.
Let us note that since $T$ is a diffeomorphism one has
$$
\ball(\mathsf{x}) \cap T^{n}\ball(\mathsf{x}) \ne \varnothing \qquad \iff \qquad \ball(\mathsf{x}) \cap T^{-n}\ball(\mathsf{x}) \ne \varnothing.
$$
We partition $\mathcal{V}_\rho$ into level sets $\mathcal{N}_{\rho}(n)$ as follows
$$
\mathcal{V}_\rho = \{\mathsf{x} \in M: \ball(\mathsf{x}) \cap T^{-n}\ball(\mathsf{x}) \ne \varnothing \text{ for some } 1 \leq n < J\}
 = \bigcup_{n=1}^{J-1} \mathcal{N}_{\rho}(n)
 $$
 where
 $$
  \mathcal{N}_{\rho}(n) = \{\mathsf{x} \in M: \ball(\mathsf{x}) \cap T^{-n}\ball(\mathsf{x}) \ne \varnothing \}.
$$
The above union is split into two collections $\mathcal{V}_\rho^1 $ and $\mathcal{V}_\rho^2$, where
\begin{equation*}
\mathcal{V}_\rho^{1} = \bigcup_{n=1}^{\floor{\mathfrak{b} J}} \mathcal{N}_{\rho}(n) \quad \text{and} \quad \mathcal{V}_\rho^{2} = \bigcup_{n=\roof{\mathfrak{b} J}}^{J} \mathcal{N}_{\rho}(n).
\end{equation*}
and where the constant $\mathfrak{b} \in (0,1)$ will be chosen below.
In order to find the measure of the total set we will estimate the measures of the two parts separately.

\vspace{3mm}

%%%%%%%%%%%%%%%%%%%%%%%%%%%%%%%%          Part 1
\noindent {\bf (I) Estimate of $\mathcal{V}_\rho^{2}$}

\vspace{1mm}

\noindent We will derive a uniform estimate for the measure of the level sets $\mathcal{N}_{\rho}(n)$ when $n > \mathfrak{b} J$.
Then
$$
\mu(\mathcal{N}_\rho(n))=\mu(T^{-n}\mathcal{N}_{\rho}(n))
\le\sum_{\zeta} \mu(T^{-n}\mathcal{N}_{\rho}(n) \cap \zeta)
$$
We will consider each of the measures $\mu(T^{-n}\mathcal{N}_{\rho}(n) \cap \zeta)$ separately by using the product form of the measures $\mu$.
By distortion of the Jacobian we obtain
\begin{eqnarray}
\mu_{\gamma^u}(T^{-n}\mathcal{N}_{\rho}(n) \cap \zeta)
&=& \frac{\mu_{\gamma^u}(T^{-n}\mathcal{N}_{\rho}(n) \cap \zeta)}
{\mu_{\gamma^u}(\zeta)} \, \mu_{\gamma^u}(\zeta)\notag\\
&\leq &c_1\omega(n)  \,\frac{\mu_{\hat\gamma^u}(T^{n}(T^{-n}\mathcal{N}_{\rho}(n) \cap \zeta))}{\mu_{\hat\gamma^u}(T^n\zeta)} \,\mu_{\gamma^u}(\zeta),  \label{level_summand}
\end{eqnarray}
where, as before, $\hat\gamma^u=\gamma^u(T^nx)$ for $x\in\zeta\cap\gamma^u$.
We estimate the numerator by finding a bound for the diameter of the set. Let the points $x$ and $z$ in
$T^{-n}\mathcal{N}_{\rho}(n)$ be such that 
$ x, z \in T^{-n}\mathcal{N}_{\rho}(n) \cap \zeta\cap\gamma^u$
for an unstable leaf $\gamma^u$. 

Note that $T^n x , T^n z \in \mathcal{N}_{\rho}(n)$, there exists $y\in B_\rho(T^n x)$ such that $ T^n y\in B_\rho(T^n x)$, thus
\[d(T^n x, x)\leq d(T^n x, T^n y )+d(T^n y, y)+d(y, x)\leq \rho+\rho+A^n d(T^n x, T^n y )\leq (2+A^n)\rho.\]
Hence
$$
d(T^n x, T^n z) \leq d(T^n x, x) + d(x,z) + d(z, T^n z)
 \leq 4 A^n \rho + d(x,z).
$$ 
We have
$$
d(x, z) \leq \diam \zeta < \delta(n)
$$
by assumption. Therefore
$$
d(T^nx, T^nz)  \leq 4A^n \rho + d(x,z) 
 \leq 4 \, A^{n} \rho + \delta(n)
$$
If we choose $\mathfrak{a}>0$ so that $\mathfrak{a}<\frac{1}{2\log A}$ then
$A^n\rho< e^{-\frac12\abs{\log\rho}^{1/2}}$. If $n\ge \mathfrak{b}\abs{\log\rho}$ for 
some $\mathfrak{b}\in(0,\mathfrak{a})$ then
$$
d(T^nx, T^nz)  \leq c_2(e^{-\mathfrak{c}'\abs{\log\rho}^{1/2}}+\delta(n))
$$
for some constant $c_2$ where $\mathfrak{c}'=\min(\frac12,\sqrt\mathfrak{b})$.
Taking the supremum over all points $x$ and $z$ yields
$$
\abs{T^n (T^{-n}\mathcal{N}_{\rho}(n) \cap \zeta\cap\gamma^u)} 
\leq c_2(e^{-\mathfrak{c}'\abs{\log\rho}^{1/2}}+\delta(n)).
$$
By assumption (V3) on the relationship between the measure and the metric
$$
\mu_{\hat\gamma^u}(T^n (T^{-n}\mathcal{N}_{\rho}(n) \cap \zeta))
\leq c_3(e^{-u_0\mathfrak{c}'\abs{\log \rho}^{1/2}}+\delta(n)^{u_0})
$$
%which implies by the product structure of $m$ that 
%$$
%\mu_{\hat\gamma^u}(T^n (T^{-n}\mathcal{N}_{\rho}(n) \cap \zeta)) 
%\leq c_3(e^{-{u_0}\mathfrak{c}'\abs{\log \rho}^{1/2}}+\delta(n)^{u_0})
%$$
Incorporating the estimate into~\eqref{level_summand} yields
$$
\mu_{\gamma^u}(T^{-n}\mathcal{N}_{\rho}(n) \cap \zeta)
\leq c_4\omega(n) (e^{-{u_0}\mathfrak{c}'\abs{\log \rho}^{1/2}}+\delta(n)^{u_0})
 \mu(\zeta),
$$
for some $c_4$.
Integrating over $d\upsilon(\gamma^u)$ and summing over $\zeta$ yields
$$
\mu(\mathcal{N}_{\rho}(n))
\le c_4\omega(n) (e^{-{u_0}\mathfrak{c}'\abs{\log \rho}^{1/2}}+\delta(n)^{u_0})
\sum_\zeta \mu(\zeta) 
\le c_5\omega(n) (e^{-{u_0}\mathfrak{c}'\abs{\log \rho}^{1/2}}+\delta(n)^{u_0})
$$
as $\sum_\zeta \mu(\zeta)=\mathcal{O}(1)$.
Consequently, if $\omega(n)$ is so that $\limsup_{n\to\infty}\frac{\log\log\omega(n)}{\log n}<\frac12$
(as can be seen from the estimates above, the value $\frac12$ can be replaced by $1$) 
then 
\begin{eqnarray}\label{Th2refPt2}
\mu(\mathcal{V}_\rho^{2}) & \leq& \sum_{n=\roof{\mathfrak{b} J}}^{J}
 \mu(\mathcal{N}_{\rho}(n))\\
 &\leq& c_5 e^{-{u_0}\mathfrak{c}'\abs{\log \rho}^{1/2}} \sum_{n=\roof{\mathfrak{b} J}}^{J}\omega(n) 
 + c_5\sum\limits_{n=\roof{\mathfrak{b} J}}^{J}\omega(n)\delta(n)^{u_0}\notag\\
 &\le& c_6(e^{-\mathfrak{c}''\abs{\log \rho}^{1/2}}+(\mathfrak{ab}\abs{\log\rho})^{-\sigma})\notag
\end{eqnarray}
for some constant $\mathfrak{c}''>0$ (and $\rho$ small enough) as $J=\lfloor\mathfrak{a}\abs{\log\rho}\rfloor$. 
As before,  $\sigma= \kappa u_0-\kappa'-1$ if $\delta(n)\sim n^{-\kappa}$ and $\omega(n)\sim n^{\kappa'}$.

\vspace{3mm}

%%%%%%%%%%%%%%%%%%%%%%%%%%%%%%%%%%%%%%    Part 2
\noindent {\bf (II) Estimate of $\mathcal{V}_\rho^{1}$}

\vspace{1mm}

We will need the following  version of Lemma~B.3 from~\cite{CC13}.

\begin{lem}\label{B.3}
Put $s_p = 2^p \, \frac{A^{n \, 2^p}-1}{A^n-1}$. Then  for every $p,k$ integers, $\rho>0$ 
$$
\left\{\mathsf{x} \in M: B_{\rho}(\mathsf{x}) \cap T^{k}B_{\rho}(\mathsf{x}) \ne \varnothing\right\}
\subset
\left\{\mathsf{x} \in M: B_{s_p \rho}(\mathsf{x}) \cap T^{k2^p}B_{s_p \rho}(\mathsf{x}) \ne \varnothing\right\}.
$$
\end{lem}
\noindent {\bf Proof.} 
Consider the case $p=1$. Let $x$ such that $\ball(\mathsf{x}) \cap T^{k}\ball(\mathsf{x}) \ne \varnothing$. This implies that there exist $z\in \ball(\mathsf{x}) \cap T^{-k}\ball(\mathsf{x})$. For any $u\in T^{k}\ball(\mathsf{x})$, there exist $v\in \ball(\mathsf{x})$ such that $T^k v=u$, thus
\[d(u,x)\leq d(u, T^kz)+d(T^kz,x)\leq d(T^k v,T^kz)+2\rho\leq (2A^k+2)\rho.\]
Therefore, $T^{k}\ball(\mathsf{x})\subset B_{(2A^k+2)\rho}(\mathsf{x})$.

One can observe that if $\ball(\mathsf{x}) \cap T^{k}\ball(\mathsf{x}) \ne \varnothing$ then $T^{k}\left(\ball(\mathsf{x}) \cap T^{k}\ball(\mathsf{x})\right) \ne \varnothing$ thus $T^{k}\ball(\mathsf{x}) \cap  T^{k}(T^{k}\ball(\mathsf{x})) \ne \varnothing$ and therefore $B_{(2A^k+2)\rho}(\mathsf{x}) \cap  T^{k}(T^{k}B_{(2A^k+2)\rho}(\mathsf{x})) \ne \varnothing$. Finally, this gives us
\[\{\mathsf{x} \in M: \ball(\mathsf{x}) \cap T^{k}\ball(\mathsf{x}) \ne \varnothing \}\subset \{\mathsf{x} \in M: B_{(2A^k+2)\rho}(\mathsf{x}) \cap T^{k}(T^{k}B_{(2A^k+2)\rho}(\mathsf{x})) \ne \varnothing \}.\]
 The general case is shown similarly.
\qed

\vspace{3mm} 

\noindent The lemma thus shows that 
 $\mathcal{N}_{\rho}(n) \subset \mathcal{N}_{s_p \rho}(2^p n)$ 
and consequently we only need to estimate 
 $\mu(\mathcal{N}_{s_p \rho}(2^p n))$.

Let us now consider the case $1 \leq n \leq \floor{\mathfrak{b}J}$ and let as in Lemma~\ref{B.3}
$s_p = 2^p \, \frac{A^{n \, 2^p}-1}{A^n-1}$.
Hence by Lemma~\ref{B.3} one has 
$\mathcal{N}_{\rho}(n) \subset \mathcal{N}_{s_p \rho}(2^p n)$.
for any $p \geq 1$, and in particular for $p(n) = \floor{\lg \mathfrak{b}J - \lg n} + 1$.
Therefore
$$
\bigcup_{n=1}^{\floor{\mathfrak{b}J}} \mathcal{N}_{\rho}(n) \subset \bigcup_{n=1}^{\floor{\mathfrak{b}J}} \mathcal{N}_{s_{p(n)} \rho}(2^{p(n)} n).
$$
Now define
\begin{equation*}
n' = n 2^{p(n)} \qquad \text{ and } \qquad \rho' = s_{p(n)} \rho.
\end{equation*}
A direct computation shows that $1 \leq n \leq \floor{\mathfrak{b}J}$ implies $\roof{\mathfrak{b}J} \leq n' \leq 2 \mathfrak{b}J$ and so
$$
\mathcal{V}_\rho^{1} = \bigcup_{n=1}^{\floor{\mathfrak{b}J}} \mathcal{N}_{\rho}(n) \subset \bigcup_{n=1}^{\floor{\mathfrak{b}J}} \mathcal{N}_{s_{p(n)} \rho}(2^{p(n)} n) \subset \bigcup_{n'=\roof{\mathfrak{b}J}}^{2 \mathfrak{b}J} \mathcal{N}_{\rho'}(n').
$$
Therefore to estimate the measure of $\mathcal{V}_\rho^1$ it suffices to find a bound for
$\mathcal{N}_{\rho'}(n')$ when $n' \geq \mathfrak{b}J$. This is accomplished by using
an argument analogous to the first part of the proof. We replace all the $n$ with $n'$
 and $\rho$ with $\rho'$.  We get for $\mathfrak{b}< 1/3$
$$
\mu(\mathcal{N}_{\rho'}(n')) \leq c_5\xi(n')(e^{-{u_0}\abs{\log \rho'}^{1/2}}  +\delta(n')^{u_0})
$$
and thus obtain an estimate similar to~\eqref{Th2refPt2}:
$$
\mu(\mathcal{V}_\rho^{1}) 
\leq \sum_{n'=\roof{\mathfrak{b}J}}^{2 \mathfrak{b}J} \mu(\mathcal{N}_{\rho'}(n'))
 \le c_8(e^{-\mathfrak{c}\abs{\log \rho'}^{1/2}}  +(\mathfrak{ab}\abs{\log\rho})^{-\sigma}).
$$
for some $\mathfrak{c}\in(0,u_0)$.

\vspace{3mm}

%%%%%%%%%%%%%%%%%%%%%%%%%%%%%%%%%%%%%%%% Part 3
\noindent {\bf (III) Final estimate}

\vspace{1mm}

\noindent Overall we obtain for all $\rho$ sufficiently small
$$
\mu(\mathcal{V}_\rho) 
\leq \mu(\mathcal{V}_\rho^{1})+\mu(\mathcal{V}_\rho^{2})
\le c_9(e^{-\mathfrak{c}\abs{\log \rho'}^{1/2}}  +(\mathfrak{ab}\abs{\log\rho})^{-\sigma})
\le C_2 \abs{\log\rho}^{- \sigma},
$$
for some $C_2$.
 \end{proof}
 
 %%%%%%%%%%%%%%%%%%%%%%%%%%%%%%%%%%%%%%%%%%%%%%
%%%%%%%%%%%%%%%%%%%%%%%%%%%%%%%%%%%%%%%%%%%%%%
%%%%%%%%%%%%%% Approx. Theorem

\section{Poisson Approximation Theorem} \label{poisson}

This section contains the abstract Poisson approximation theorem which establishes the distance
between sums of $\{0,1\}$-valued dependent random variables $X_n$ and a random variable that
is Poisson distributed. It is used in Section~\ref{set_up_T1} in the proof of Theorem~1 and compares
 the number of occurrences in a finite time interval with the number of occurrences in the same interval
 for a Bernoulli process $\{\tilde{X}_n:n\}$.

\begin{thrm}\cite{CC13} \label{helperTheorem}
Let $(X_n)_{n \in \mathbb{N}}$ be a stationary $\{0,1\}$-valued process and $t$ a positive parameter.
 Let $S_a^b = \sum_{n=a}^b X_n$ and define $S:=S_1^N$ for convenience's sake where
 $N=\floor{t/\epsilon}$ and $\epsilon = \mathbb{P}(X_1 = 1)$. Additionally, let $\nu$ be the Poisson distribution measure with mean $t>0$. Finally, assume that $\epsilon < \frac{t}{2}$. Then there exists a constant $C_3$ such that for any $E \subset \nats$, and $2 \leq \Delta < N$ we have
\begin{equation*}
\abs{\mathbb{P}(S \in E) - \nu(E)} \leq C_3 \#\{E \cap [0,N]\} \; (N(\mathcal{R}_1 + \mathcal{R}_2) + \Delta \epsilon)
\end{equation*}
where,
\begin{align*}
\mathcal{R}_1 &= \sup_{\substack {0 <j<N-\Delta \\ 0<q<N-\Delta-j}} \{ \abs{\mathbb{P}(X_1=1 \land S_{\Delta+1}^{N-j}=q) - \epsilon \, \mathbb{P}(S_{\Delta+1}^{N-j}=q)} \} \\
\mathcal{R}_2 &= \sum_{n=2}^\Delta \mathbb{P}(X_1=1 \land X_n=1).
\end{align*}
\end{thrm}

\begin{proof} Let $(\tilde{X}_n)_{n \in \mathbb{N}}$ be a sequence of independent, identically distributed random variables taking values in $\{0,1\}$, constructed so that $\mathbb{P}(\tilde{X}_1=1)=\epsilon$. Further assume that the $\tilde{X}_n$'s are independent of the $X_n$'s. Let $\tilde{S}=\sum_{n=1}^N \tilde{X}_n$. Then
\begin{align*}
\abs{\mathbb{P}(S \in E) - \nu(E)}
& \leq \abs{\mathbb{P}(S \in E) - \mathbb{P}(\tilde{S} \in E)} + \abs{\mathbb{P}(\tilde{S} \in E) - \nu(E)} \\
& \leq \! \! \sum_{k \in E \cap [0,N]} \! \! \abs{\mathbb{P}(S=k)-\mathbb{P}(\tilde{S}=k)} + \sum_{k=0}^{\infty} \, \Abs{\mathbb{P}(\tilde{S}=k) - \frac{t^k}{k!} e^{-t}}
\end{align*}
Thanks to \cite{AGG89} we can bound the second sum using the estimate
\begin{equation} \label{GoldArratia}
\sum_{k=0}^{\infty} \Abs{\mathbb{P}(\tilde{S}=k) - \frac{t^k}{k!} e^{-t}} \leq \frac{2t^2}{N}.
\end{equation}
For summands of the remaining term we utilize the proof of Theorem 2.1 from \cite{CC13}
according to which for every $k \leq N$,
\begin{equation*}
\abs{\mathbb{P}(S=k)-\mathbb{P}(\tilde{S}=k)}
\leq 2N(\mathcal{R}_1+\mathcal{R}_2+\Delta \epsilon^2) + 4\Delta \epsilon.
\end{equation*}
As $N \leq t/\epsilon$ this becomes
\begin{equation} \label{EstSAndTilde}
\abs{\mathbb{P}(S=k)-\mathbb{P}(\tilde{S}=k)}\leq 6 \, t \, (N(\mathcal{R}_1+\mathcal{R}_2) + \Delta \epsilon).
\end{equation}
Combining~\eqref{GoldArratia} and~\eqref{EstSAndTilde} yields
\begin{align*}
\abs{\mathbb{P}(S \in E) - \nu(E)} &\leq \sum_{k \in E \cap [0,N]} \hspace{-0.2cm} \abs{\mathbb{P}(S=k)-\mathbb{P}(\tilde{S}=k)} \, + \, \frac{2t^2}{N} \\
&\leq \sum_{k \in E \cap [0,N]} \hspace{-0.2cm} 6 \, t \, (N (\mathcal{R}_1 + \mathcal{R}_2) + \Delta \epsilon) \, + \, \frac{2t^2}{t / \epsilon -1} \\
&\leq 6 \, t \, \#\{E \cap [0,N]\} \, (N(\mathcal{R}_1+\mathcal{R}_2) + \Delta \epsilon) + 4 \, t\epsilon \\
&\leq C_3 \#\{E \cap [0,N]\} \; (N(\mathcal{R}_1 + \mathcal{R}_2) + \Delta\epsilon)
\end{align*}
for some $C_3<\infty$.
\end{proof}

%%%%%%%%%%%%%%%%%%%%%%%%%%%%%%%%%%%%%%%%%%%%
%%%%%%%%%%%%%%%%%%%%%%%%%%%%%%%%%%%%%%%%%%%%
%%%%%%%%%%%%%%%%%%%%%%%%%%%%%%%%%%%%%%%%%%%%
%%%%%%%%%%       SECTION: Extreme value law
%%%%%%%%%%%%%%%%%%%%%%%%%%%%%%%%%%%%%%%%%%%%

\section{Proof of Theorem~\ref{EVL}}

In~\cite{L80} Leadbetter et al gave two conditions called $D$ and $D'$, under which $a_n(M_n -b_n) \to G$ is equivalent to $a_n(\hat{M}_n -b_n) \to G$. Recall that $\hat{M}_n$ is the maxima of the independent, stationary process $\{\hat{X}_n\}$. Later Freitas et al~\cite{FF08} replaced condition $D$ by $D_2$ and obtained the same result. To state the conditions we put $u_n = v/a_n + b_n$ for $v\in \mathbb{R}$ and sequences $a_n, b_n$.
\begin{cdn} $D_2(u_n)$~\cite{FF08} 
We say condition  $D_2(u_n)$ holds if for any integers $l,t$ and $n$
$$
|\mu(X_0>u_n, M_{t,l}<u_n) - \mu(X_0>u_n)\mu(M_l<u_n)| \le \gamma(n,t)
$$
where $\gamma(n,t)$ is a non-increasing sequence in $t$ for every $n$ and satisfies $\gamma(n,t_n) = o(\frac{1}{n})$ for some sequence $t_n = o(n)$, $t_n \to \infty$.
\end{cdn}
\begin{cdn} $D'(u_n)$~\cite{FF08} We say condition $D'(u_n)$ holds if 
$$
\lim\limits_{k\to\infty}\limsup_{n} n\cdot\sum_{j=1}^{\floor{n/k}}\mu(X_0>u_n, X_j>u_n ) =0
$$
\end{cdn}

Below we will verify both conditions for Type I observable, i.e.\ $g$ with \\ $g^{-1}\left(g(\frac{1}{n})+yp(g(\frac{1}{n}))\right)=(1+\varepsilon_n)\frac{e^{-y}}{n}$. The other two cases follow similarly.

\subsection{Condition $D_2(u_n)$}
First we show $D_2(u_n)$. Put $u_n(y) = g\left(\frac{1}{n}\right) + yp\left(g\left(\frac{1}{n}\right)\right)$, then
$$
\{X_0>u_n\} = B_{l(g^{-1}(u_n))}(z);
$$
here $l(y) = \inf\{r>0: \mu(B_r(z))\ge y\}$. Since $\mu(B_{l(y)}(z)) = y$, by Assumption (II) we get
$$
C y^{1/d_0}\le l(y) \le C' y^{1/d_1}
$$
for some constant $C$ and $C'$. In particular we have

\begin{equation}
 C n^{-1/d_0} \le l(g^{-1}(u_n)) \le C' n^{-1/d_1}.
\end{equation}
Here both constants depend on $y$.

To simplify notations we write $r_n = l(g^{-1}(u_n))$ and omit $z$. We approximate the indicator function of $\{Y_0>u_n\} =B_{r_n}$ by  Lipschitz functions $\phi(x)$ and $\tilde{\phi}(x)$ as in the proof of $\mathcal{R}_1$. The same estimate as in Section~\ref{est_R1_section} yields
\begin{eqnarray*}
\left|\mu(Y_0>u_n, M_{t,l}<u_n) - \mu(Y_0>u_n)\mu(M_l<u_n)\right| \hspace{-5cm}&&\\
&=& \left|\int \ind_{B_{r_n}} \ind_{\{M_{t,l}<u_n\}}\,d\mu - \int \ind_{B_{r_n}}\,d\mu \int\ind_{\{M_{t,l}<u_n\}}\,d\mu \right|\\
&\le& c_2\left(\frac{\lambda(t/2)}{\delta r_n}
+\mu(B_{r_n +\delta r_n} \setminus B_{r_n -\delta r_n})\right)
+\mathcal{O}(r_n^{v(\kappa\eta-1)-\beta})\mu(B_{r_n}).
\end{eqnarray*}
Putting $\delta r_n = r_n^{w}$, $w>1$, and $t=n^{v}$ with $0<v <1 $ gives
\begin{eqnarray*}
\gamma(n,t)& = &\left|\mu(Y_0>u_n, M_{t,l}<u_n) - \mu(Y_0>u_n)\mu(M_l<u_n)\right| \\
&\le& \mathcal{O}(1)\left(n^{-vp+w/d_1}+n^{-1-(w\eta-\beta)/d_1}\right) 
+ \mathcal{O}(r_n^{v(\kappa\eta-1)-\beta})\frac1n
\end{eqnarray*}
as $\mu(Y_0>u_n)=\mu(B_{r_n})=\mathcal{O}(1/n)$.
In order that $n\gamma(n,t)\to 0$ we need $1-vp+w/d_1<0$, $-(w\eta-\beta)/d_1<0$ and 
$v(\kappa\eta-1)-\beta>0$. We choose $w > \beta/\eta $  and $v $ close to 1. 
In the case when $\lambda$ decays polynomially with power $p$, $1-vp+w/d_1<0$ is satisfied if 
$p >\frac{\beta}{\eta d_1}+1$.
  
\subsection{Condition $D'(u_n)$} Notice that 
$$
\sum_{j=1}^{\floor{n/k}}\mu(X_0>u_n, X_j>u_n )  = \sum_{j=1}^{\floor{n/k}} \mu(B_{r_n}\cap T^{-j}B_{r_n}).
$$
This is exactly $\mathcal{R}_2$ in Section~\ref{set_up_T1} with $\Delta = \floor{n/k}$. We split the sum 
into two parts as follows:
$$
\sum_{j=1}^{\floor{n/k}} \mu(B_{r_n}\cap T^{-j}B_{r_n}) = \sum_{j=J}^{\floor{n/k}} \mu(B_{r_n}\cap T^{-j}B_{r_n})+\sum_{j=1}^{J} \mu(B_{r_n}\cap T^{-j}B_{r_n})
$$
with $J = \floor{\mathfrak{a} \, \abs{\log r_n}}$.  By~(\ref{R_2}) we get
$$
\sum_{j=J}^{\Delta} \mu(B_{r_n}\cap T^{-j}B_{r_n}) \le c_5\sum_{j=J}^{\Delta-1} \omega(j)\delta(j)^{u_0}\mu(B_{r_n}),
$$
where $\mu(B_{r_n}) = \mathcal{O}(\frac1n)$. For the second term we restrict to points 
$z\not\in\mathcal{V}_{r_n}$ which implies that $B_{r_n}\cap T^{-j}B_{r_n}=\varnothing$ for 
$j=1,\dots, J-1$, where by Section~\ref{VeryShortReturns}
$$
\mu(\mathcal{V}_{r_n}) \le C_2 \abs{\log r_n}^{- \sigma},
$$
with $\sigma=\kappa u_0-\kappa'-1>0$.

To finish the proof we use the maximal function technique by Collet in~\cite{C01}. 
For this purpose we fix some $0< \xi<\theta<1$ and define the set 
$$
F_k = \left\{\mu(B_{r_{\exp(k^\xi)}} \cap \mathcal{V}_{r_{\exp(k^\xi)}}) \ge \mu(B_{r_{\exp(k^\xi)}}) \cdot k^{-\theta}\right\}
$$
and 
$$
M_r(x) = \sup_{s>0}\frac{1}{\mu(B_s(x))}\int_{B_s(x)}\ind_{\mathcal{V}_r}(y)\,d\mu(y)
$$
where $\mathcal{V}_r = \{\mathsf{x} \in M: B_{r} \cap T^nB_{r} \ne \varnothing \text{ for some } 1 \leq n < J\}$ as before.
Since
$$
F_k \subset \{M_{r_{\exp(k^\xi)}} \ge k^{-\theta}\},
$$
we conclude that 
$$
\mu(F_k) 
\le \mu(M_{r_{\exp(k^\xi)}} \ge k^{-\theta}) 
\le \frac{\mu(\ind_{\mathcal{V}_{r_{\exp(k^\xi)}}(y)})}{k^{-\theta}}
\le k^{-(\xi\sigma -\theta)}.
$$
If $\xi\sigma -\theta>1$ we get $\sum_k\mu(F_k) < \infty$ and thus by Borel-Cantelli there exist 
$N(x)$ for almost every $x$ such that $x \notin F_k$ for all $k>N(x)$. 
For every $n$, %recall that $Cn^{-1/d_0} \le r_n \le C'n^{-1/d_1}$. 
we choose $k$ such that 
$$
\exp(k^\xi)\le n < \exp((k+1)^\xi),
$$
we have $r_{\exp((k+1)^\xi)}\le r_n \le r_{\exp(k^\xi)}$. As a result
$$
B_{r_n}\cap T^{-j}B_{r_n} 
\subset B_{r_{\exp(k^\xi)}}\cap T^{-j} B_{r_{\exp(k^\xi)}}
\subset B_{r_{\exp(k^\xi)}}\cap \mathcal{V}_{r_{\exp(k^\xi)}}
$$
for every $j < J $. Therefore
\begin{eqnarray*} 
n \cdot \sum_{j=1}^{J}  \mu(B_{r_n} \cap T^{-j}B_{r_n})
& \le& n \cdot \sum_{j=1}^{J} \mu( B_{r_{\exp(k^\xi)}}\cap \mathcal{V}_{r_{\exp(k^\xi)}})\\
&\le& C\exp((k+1)^\xi) \cdot (k+1)^\xi \mu(B_{r_{\exp(k^\xi)}}) \cdot k^{-\theta}\\
&\le& C \frac{ \exp((k+1)^\xi)}{\exp(k^\xi)}\cdot k^{\xi -\theta} \to 0
\end{eqnarray*}
since $\xi - \theta <0$. If $\sigma=\kappa u_0-\kappa'-1>2$ then we can choose $0<\xi<\theta<1$ both close to $1$ 
so that $\sigma\xi-\theta>1$.

%%%%%%%%%%%%%%%%%%%%%%%%%%%%%%%%%%%%%%%%%%%%
%%%%%%%%%%%%%%%%%%%%%%%%%%%%%%%%%%%%%%%%%%%%
%%%%%%%%%%%%%%%%%%%%%%%%%%%%%%%%%%%%%%%%%%%%
%%%%%%%%%%       SECTION: EXAMPLE
%%%%%%%%%%%%%%%%%%%%%%%%%%%%%%%%%%%%%%%%%%%%

\section{Example}

As an example we consider the Manville-Pommeau map on the unit interval.
It is given by 
$$
Tx=\begin{cases}
x+2^{1+\alpha}x^{1+\alpha}& \mbox{ for } 0\le x\le\frac12\\
2x-1 & \mbox{ for } \frac12< x\le 1
\end{cases},
$$
where $\alpha\in(0,1)$ is a parameter. In this case $T$ has an absolutely 
continuous invariant measure $\mu$ whose density is $h(x)\sim x^{-\alpha}$.
The return times distribution has previous been shown to be Poissonia in~\cite{HSV}.
Also, an inducing argument was used in~\cite{BSTV03} to show that the 
first return time is almost surely exponentially distributed. Here we apply 
our main theorem to give a short argument to deduce the Poisson 
distribution of entry times. For this we also rely on a result of Hu~\cite{Hu04}
which proves that the transfer operator converges at a polynomial rate and 
thus that the decay of correlations (as in Assumption~(I)) is polynomial.

There is a sequence of points $a_n, n=0,1,\dots$ which decreases to $0$ so that 
$T_0=\frac12$ and $Ta_{n+1}=a_n$ for all $n$. If we put $I_n=(a_{n+1},a_n]$,
then all the intervals $I_n$ are pairwise disjoint and 
satisfy  $TI_{n+1}=I_n$ for all $n$ and $\bigcup_nI_n=(0,\frac12)$.
Moreover $a_n\sim n^{-\gamma}$, where $\gamma=\frac1\alpha$ is larger 
than $1$. Since $h(x)\sim x^{-\alpha}$ one has  $\mu(I_n)\sim n^{-\gamma}$
and $\mu(J_n)=n^{1-\gamma}$ where $J_n=\bigcup_{j=n}^\infty I_j$ is 
a punctured neighbourhood of $0$.

The two elements $(0,\frac12]$ and $(\frac12,1]$ cover the entire 
unit interval and denote by $\mathscr{I}_n$ the inverse branches
of $T^n$.

Put $A_k=(a_{k+1},a_0]=\bigcup_{j=0}^kI_k$ and let $\mathscr{I}_n$ be 
the inverse branches of $T^n$. If $\hat\zeta_\varphi$ is an $n$-cylinder, 
that is a preimage of either $A_k$ or $(\frac12,1]$ under the inverse
branch $\varphi\in\mathscr{I}_n$ then the distortion of
$DT^n$ on $\hat\zeta_\varphi$ is bounded by 
$c_1\!\left(\frac1k+\frac1n\right)^{-\gamma(\gamma+1)}$
for some constant $c_1$.  In particular if we choose $k=n^\theta$ for some 
$\theta\in[0,1)$ then we can put $\omega(n)=c_2n^{\theta\gamma(\gamma+1)}$ 
for some $c_2$.

We now can nearly use the theorem for higher order returns, but let us 
remark that  if $\mathsf{x}\in (0,1)$ then for $n$ large enough we have
that $\mathsf{x}\in A_{n^\theta}\cup(\frac12,1]$. If we proceed as in the 
estimate of the term $\mathcal{R}_2$ we obtain
$$
T^{-j}\ball\cap\ball\subset\bigcup_{\zeta:\zeta\cap\ball\not=\varnothing}\zeta
=\mathscr{P}_1\cup\mathscr{P}_2
$$
where  the union is over $j$-cylinders $\zeta$ and
$$
\mathscr{P}_1=\bigcup_{\zeta:\zeta\cap\ball\not=\varnothing}T^{-j}\ball\cap\hat\zeta,
\hspace{2cm}
\mathscr{P}_2=\bigcup_{\zeta:\zeta\cap\ball\not=\varnothing}T^{-j}\ball\cap\zeta\setminus\hat\zeta.
$$
The first set is estimated as before in the main theorem. 
For the second term notice that 
$$
\mathscr{P}_2
=\bigcup_{A\in\mathcal{A}}
\bigcup_{\varphi\in\mathscr{I}_j:\varphi(A)\cap\ball\not=\varnothing}
T^{-j}\ball\cap\varphi(A\setminus\hat{A})
=\bigcup_{\varphi\in\mathscr{I}_j:\varphi(A_0)\cap\ball\not=\varnothing}
T^{-j}\ball\cap\varphi(J_{n^\theta})
$$
where $\mathcal{A}=\{(0,\frac12],(\frac12,1]\}$ and $A=A_{j^\theta}$ if $A=(0,\frac12]$
and otherwise $\hat{A}=A$. Hence
$$
\mathscr{P}_2
=\bigcup_{\varphi\in\mathscr{I}_j:\varphi(A_0)\cap\ball\not=\varnothing}\varphi(\ball\cap J_{n^\theta})
$$
which is empty for $n$ large enough, i.e.\ so that $a_{n^\theta}<\mathsf{x}$.

By~\cite{Hu04} Proposition~5.2 one has that the correlations decay polynomially
at the rate of $\gamma-1$, that is  $\lambda(k)=c_3k^{1-\gamma}$ ($p=\gamma-1$)
for some $c_3$.
The dimensions here are $d_0=d_1=1$ and the annulus condition is satisfied with
$\eta=1$ and $\beta=0$. Similarly, $u_0=1$.
In order to get the contraction rate consider the `worst' case for the 
contraction, when the partition element $(0,\frac12]$ is $n$ times
mapped by the inverse branch that contains the parabolic branch.
Its  image is then $(0,a_n)$ and therefore $\delta(n)=a_n\sim n^{-\gamma}$.
Hence $\kappa=\gamma$. Since $\gamma>1$ the conditions of the theorem
are satisfied since we can choose $\theta>0$ arbitrarily close to $0$. 
For any $\sigma<\gamma$ one can choose $\theta>0$ so that 
$\sigma\le\gamma-\theta\gamma(\gamma-1)$ and therefore we obtain
the following result.

\begin{cor}
Let $T$ be the Manneville-Pommeau map for the parameter $\alpha\in(0,1)$.
Let $\mu$ be the invariant absolutely continuous probability measure.
Then for any $\sigma<\gamma=\frac1\alpha$ one has
$$
\mathbb{P}(\xi_{\rho,\mathsf{x}}=r)=e^{-t}\frac{t^r}{r!}+\mathcal{O}(\abs{\log \rho}^{-\sigma})
$$
for all $\mathsf{x}\not\in \mathcal{V}_\rho(\mathfrak{a})$ for some positive $\mathfrak{a}$.
Moreover $\mu(\mathcal{V}_\rho(\mathfrak{a}))=\mathcal{O}(\abs{\log \rho}^{-\sigma})$.
\end{cor}

%%%%%%%%%%%%%%%%%%%%%%%%%%%%%%%%%%%%%%%%%%%%
%%%%%%%%%%%%%%%%%%%%%%%%%%%%%%%%%%%%%%%%%%%%
%%%%%%%%%%%%%%%%%%%%%%%%%%%%%%%%%%%%%%%%%%%%
%%%%%%%%%%       BIBLIOGRAPHY
%%%%%%%%%%%%%%%%%%%%%%%%%%%%%%%%%%%%%%%%%%%%

\end{document}